\documentclass[11pt]{article}
\usepackage[T1]{fontenc}
\usepackage{amsmath,amsthm,amssymb,mathrsfs,mathtools}
\usepackage{enumerate}
\usepackage{graphicx,booktabs,microtype,xcolor,tikz,placeins}
\usetikzlibrary{arrows.meta,calc}
\usepackage[breaklinks,colorlinks,linkcolor=black,citecolor=blue,urlcolor=blue]{hyperref}
\usepackage{nameref}
\usepackage[nameinlink,noabbrev]{cleveref}
\usepackage[margin=27mm]{geometry}

\newtheorem{theorem}{Theorem}[section]
\newtheorem{lemma}[theorem]{Lemma}
\newtheorem{corollary}[theorem]{Corollary}
\newtheorem{proposition}[theorem]{Proposition}
\theoremstyle{definition}

\theoremstyle{remark}

\numberwithin{equation}{section}
\allowdisplaybreaks[2]

\newcommand{\R}{\mathbb{R}}
\newcommand{\dd}{\mathop{}\!\mathrm d}
\newcommand{\II}{\mathrm{II}}
\newcommand{\Ric}{\operatorname{Ric}}

\newcommand{\diver}{\operatorname{div}}

\begin{document}

\title{Counterexamples to Wang's Conjecture}

\author{
Ruiqi Jiang
\footnote{School of Mathematics, Hunan University, Changsha 410082, China (jiangruiqi@hnu.edu.cn).}
\qquad \and
Linlin Sun
\footnote{School of Mathematics and Computational Science, Xiangtan University, Xiangtan 411105, China (sunll@xtu.edu.cn)}
}

\date{}
\maketitle

\begin{abstract}
In this paper, we give a negative answer to Wang's conjecture. For every $n\geq 3$, there exist $\varepsilon>0$ and $\delta \in (0,1)$ such that, for all $q$ and $\lambda$ satisfying
\[
\frac{n}{n-2}-\varepsilon<q\le \frac{n}{n-2},
\qquad
\frac{1}{q-1}\cdot \delta<\lambda\le \frac{1}{q-1},
\]
there exists a bounded Euclidean domain \(\Omega\subset\mathbb{R}^n\) with principal curvatures bigger than \(1\) for which the following nonlinear Robin problem admits a nonconstant positive solution:
\begin{align*}
	\begin{cases}
		\Delta u=0, & \text{in}\ \Omega,\\[2mm]
		\dfrac{\partial u}{\partial\nu}+\lambda u=u^q, & \text{on}\ \partial\Omega.
	\end{cases}
\end{align*}
\end{abstract}



\section{Introduction and main results}
Let $(M^n,g)$ be a compact connected Riemannian manifold of dimension $n$ with smooth boundary,
and let $\nu$ be the outward unit normal.  We consider positive solutions of
\begin{equation}\label{eq:wang-problem-M}
	\begin{cases}
		\Delta_g u=0,&\text{in }M, \\
		\partial_\nu u+\lambda u=u^q,&\text{on }\partial M,
	\end{cases}	
\end{equation}
where 
\begin{align}\label{ineq:condition-q-lambda}
	1<q\leq \frac n{n-2}, \qquad 0<\lambda\le\frac1{q-1}.
\end{align}
Wang (\cite[Conjecture~1]{Wang2021}) conjectured that if
\begin{equation}\label{eq:wang-assumptions}
	\Ric_g(M)\geq 0,\qquad h\geq g|_{\partial M},	
\end{equation}
where $h$ is the second fundamental form of $\partial M$ defined by $h(X,Y)=g(\nabla_X \nu,Y)$, $\forall X,Y\in \mathfrak{X}(\partial M)$, then any positive solution $u$ of \eqref{eq:wang-problem-M} must be constant $\lambda^{1/(q-1)}$ unless $q=\frac{n}{n-2}$ and $\lambda=\frac{1}{q-1}=\frac{n-2}{2}$, $M$ is isometric to closed unit ball $\overline{B^n}$ in $\mathbb{R}^n$ and $u$ takes the form
\begin{align}
	u_\xi(x)=\left[\frac{2}{n-2}\,\frac{1-|\xi|^2}{1+|\xi|^2|x|^2-2x\cdot \xi}\right]^{(n-2)/2}
\end{align}
for some $\xi \in B^n$.

For $n=2$, Wang's conjecture has been completely proved in \cite{GuoHangWang2021}. For $n\geq 3$ and $M=\overline{B^n}$, Guo and Wang (\cite{GuoWang2020}) comfirmed the uniqueness for $q\in (1,\frac{n}{n-2})$ and $\lambda \in (0,\frac{n-2}{2})$, and after some Liouville results (\cite{LinOu2023,Ou2024}), then Gu and Li (\cite{GuLi2025}) completely proved it for the full range \eqref{ineq:condition-q-lambda}. On general Riemannian manifolds, Cai (\cite{Cai2026}) only obtained some partial positive results under $\mathrm{Ric}_g(M)\geq 0$, while Guo-Hang-Wang (\cite{GuoHangWang2021}) verified some cases under $\mathrm{Sec}_g(M)\geq 0$. 

Throughout the paper, we take
\begin{align*}
	a=\frac{n-2}{2},\quad q_*=\frac{n}{n-2},
\end{align*}
and denote by $\sigma_1(M)>0$ the first positive Steklov eigenvalue of $M$, unless otherwise stated.

One can check that Wang's conjecture cannot hold for general Riemannian domains. In fact, there is a deep relationship between Wang's conjecture and the Steklov eigenvalue, as follows. For $q\in (1,q_*]$ and $\lambda \in (0,\frac{1}{q-1}]$, set
\begin{align}\label{eq:Q-functional}
	Q_{q,\lambda}(u)=\frac{\displaystyle\int_M |\nabla u|^2\,\mathrm{d}V+\lambda \int_{\partial M} u^2\,\mathrm{d}A}{\left(\displaystyle\int_{\partial M}|u|^{q+1}\,\mathrm{d}A \right)^{2/(q+1)}},
\end{align}
where $dV$ and $dA$ are volume form of $M$ and $\partial M$ respectively. After multiplication by a positive constant, each critical points of $Q_{q,\lambda}$ on $H^1(\Omega)$ satisfies \eqref{eq:wang-problem-M} weakly. 

Let $c>0$ be a constant and take $\phi\in C^1(M)$ with $\int_{\partial M} \phi \,dA=0$. Define
$$
F(s)=\frac{Q_{q,\lambda}(c+s\phi)}{Q_{q,\lambda}(c)},
$$
then
\begin{align}\label{eq:log-Fs}
	\log F(s)=\frac{s^2}{\lambda c^2|\partial M|}
	\left( \int_M |\nabla\phi|^2\,\mathrm{d}V-\lambda(q-1)\int_{\partial M}\phi^2\,\mathrm{d}A \right)
	+O(s^3).
\end{align}
If $\phi_1$ is a Steklov eigenfunction with respect to $\sigma_1(M)$, then we normolize $\phi_1$ by $\phi_*:=\phi_1 - \overline{\phi_1}$ where $\overline{\phi_1}:=\frac{1}{|\partial M|} \int_{\partial M} \phi_1 \,dA$,
such that $\int_{\partial M}\phi_* \,dA=0$, and obtain
\begin{align*}
	\int_M |\nabla \phi_*|^2\,dV=\sigma_1(M) \int_{\partial M} \phi_*^2 dA.
\end{align*}
Substitute the above equality into \eqref{eq:log-Fs}, we have
\begin{align*}
	\log F(s)=\frac{s^2}{\lambda c^2|\partial M|} \bigg(\sigma_1(M)-\lambda(q-1)\bigg) 
	\int_{\partial M}\phi_*^2\,\mathrm{d}A
	+O(s^3).
\end{align*}
Hence, if $\sigma_1(M)<\lambda(q-1)\leq 1$, then for any positive constant $c$, there holds  $\log F(s)<0$, i.e., 
\begin{align*}
	Q_{q,\lambda}(c+s\phi_*)<Q_{q,\lambda}(c)
\end{align*}
for sufficiently small $s$, which implies the constant solution of \eqref{eq:wang-problem-M} is not a local minimizer. It follows that the minimizers of \eqref{eq:Q-functional} subject to $\int_{\partial M} |u|^{q+1}\, dA=1$ for $q\in (1,\frac{n}{n-2})$ is not a constant. Note that compactness of the trace embedding $H^1(M) \hookrightarrow L^{q+1}(\partial M)$ for $q\in (1,\frac{n}{n-2})$ and $\lambda >0$ guarantee the existence of a nonnegative minimizer $u$. By the classic elliptic estimates and strong maximum principle, we can obtain $u\in C^\infty(M)$ and $u>0$ in $M$. Therefore we get the following result:
\begin{proposition}\label{prop:sigma-1-less-1}
	If $\sigma_1(M)<1$, then for every
	\begin{align}\label{ineq:Ricci-counterexample-p-lambda}
		1<q<\frac{n}{n-2},\quad \frac{\sigma_1(M)}{q-1}<\lambda\leq \frac{1}{q-1},
	\end{align}
	then problem \eqref{eq:wang-problem-M} has a positive nonconstant smooth solution.	
\end{proposition}

For $n\geq 3$, Sun-Wang-Wang (\cite{SunWangWang2026}) showed the existence of Riemannian metric $g$ on $\overline{B^n}$ such that $(\overline{B^n},g)$ has positive Ricci curvature, all principal curvatures of $\partial M$ are strictly larger than $1$, and $\sigma_1(M)<1$ (see also \cite{GirouardHeliere2026} for $n=3$). Hence, it follows from Proposition \ref{prop:sigma-1-less-1} that the counterexample of Wang's conjecture under nonnegative Ricci curvature exists.
\begin{theorem}\label{thm:ricci-main}
	For $n\geq 3$, there exist $(M^n,g)$ with positive Ricci curvature, all principal curvatures of $\partial M$ are strictly larger than $1$, and $\sigma_1(M)<1$ such that problem \eqref{eq:wang-problem-M} satisfying \eqref{ineq:Ricci-counterexample-p-lambda} has a positive nonconstant smooth solution on $M$.
\end{theorem}

Although the Escobar conjecture was disproved in the general Riemannian case, the Euclidean case was proved by Xia and Xiong \cite{XiaXiong2024}. Namely, they proved that \(\sigma_1(M)\ge 1\) for a Euclidean domain whose principal curvatures satisfy \(\kappa_i\ge 1\). As mentioned above, Gu and Li \cite{GuLi2025} proved that Wang's conjecture is true for the Euclidean ball. A natural question is the following:

\begin{quote}
	Does Wang's conjecture hold true for Euclidean convex domains?
\end{quote}

In this paper, we give a negative answer to this question. More precisely, we give negative answers to Wang's conjecture for \((q,\lambda)\) near the critical case \(\left(\frac{n}{n-2}, \frac{n-2}{2}\right)\). That is, we have the following:

\begin{theorem}\label{thm:flat-main}
	For $n\ge3$, there exist $\varepsilon>0$ and $\delta\in (0,1)$ such that for every
	$q_*-\varepsilon<q\leq q_*$ and $\frac{1}{q-1}\cdot\delta <\lambda \leq \frac{1}{q-1}$, there exist a bounded real-analytic strictly
	convex domain $\Omega_*\subset\R^n$ who admits $\sigma_1(M)>1$ and all principal curvatures
	\begin{align*}
		\min_{x\in\partial\Omega,\,1\le i\le n-1}\kappa_i(x)>1,
	\end{align*}
	and nonconstant positive function
	$u\in C^\infty(\overline{\Omega_*})$ satisfying
	\begin{equation}\label{eq:wang-problem-Omega}
		\begin{cases}
			\Delta u=0\quad\text{in } \Omega_*, \\
			\partial_\nu u+\lambda u=u^q\quad\text{on }\partial \Omega_*.
		\end{cases}	
	\end{equation}
\end{theorem}

The construction is based on an exact solution, not on a bifurcation from
the constant branch.
In Section \ref{sec:explicit-solution}, we obtain the initial Lipschitz convex domain $\Omega$ in $\mathbb{R}^n$ by intersecting $n+1$ unit balls whose centers lie on the
sphere of radius $2$;
on every spherical face of $\Omega$ the same Newtonian potential $(3a)^a |x|^{n-2}$ satisfies the critical Robin equation;
due to $\sigma_1(\Omega)>1$, Kelvin involution identifies the
linearized kernel at the Newtonian solution with the Steklov eigenspace at
eigenvalue $1$, and therefore proves its nondegeneracy.
In Section \ref{sec:solution-approximating-domain},  $\Omega$ is approximating by bounded real-analytic strictly convex domains $\Omega_\tau$ from the inside by a log-sum-exp defining function; 
after pulling the Robin problem on $\Omega_\tau$ back to $\Omega$, uniform boundary
H\"older estimates give compactness of the resolvents;  finally, we apply local
Leray--Schauder degree to obtain the nondegenerate solution on the $\Omega_\tau$ and nearby exponents.

\section{A Lipschitz domain and an explicit nonconstant solution}
\label{sec:explicit-solution}

In this section, we construct a strictly convex domain by intersecting $n+1$ unit balls whose centers lie on the sphere of radius $2$ centered at origin. 

We now specify the centers of $n+1$ unit balls in $\mathbb{R}^n$. Let 
$$
E=\left\{y\in\R^n:\sum_{i=1}^n y_i=0\right\}\cong\R^{n-1},
$$
and 
$$
\omega_j=\sqrt{\frac{n}{n-1}}\left(e_j-\frac{1}{n}(1,\dots,1)\right)\in E, \quad j=1,\cdots,n,
$$
then there holds 
\begin{align}\label{eqn:Flat-1-1}
	|\omega_j|=1, \quad \omega_i \cdot \omega_j =-\frac{1}{n-1}, \quad 1\leq i\neq j\leq n.	
\end{align}
Set $$\omega_0 =-\omega_1.$$ Then
\begin{align}\label{eqn:Flat-1-2}
	|\omega_0|=1, \quad \omega_0 \cdot \omega_1=-1, \quad \omega_0 \cdot \omega_j=\frac{1}{n-1}, \quad j=2,\cdots,n.
\end{align}
We identify $E\oplus \mathbb{R}$ as $\mathbb{R}^n$, and the $n+1$ centers of unit balls in $\mathbb{R}^n$ are
$$
\begin{aligned}
	c_0(\eta)&=\left((\eta+\eta^2)\omega_0,
	\sqrt{4-(\eta+\eta^2)^2}\right),\\
	c_j(\eta)&=\left(\eta\omega_j,\sqrt{4-\eta^2}\right),
	\qquad 1\le j\le n
\end{aligned}
$$
where $\eta\in (0,1)$. It follows that
\begin{align*}
	|c_j(\eta)|=2, \quad j=0,1,\cdots,n,
\end{align*}
and all centers converge to $2e_n$ as $\eta \rightarrow 0$.
Set
\begin{align*}
	K_\eta:=\bigcap_{j=0}^{n}B_1(c_j(\eta)).
\end{align*}

\begin{lemma}\label{lem:clustered-faces}
	There is $\eta_0(n)>0$ such that, for any $\eta\in (0,\eta_0(n))$, there holds
	\begin{enumerate}[(i)]
		\item $K_\eta$ has nonempty interior and
		$0\notin\overline K_\eta$;
		\item every sphere $\partial B_1(c_j(\eta))$ contributes a nonempty
		relatively open face to $\partial K_\eta$;
		\item the centers $c_0(\eta),\ldots,c_n(\eta)$ are affinely independent.
	\end{enumerate}
	Moreover, $K_\eta$ converges to $B_1(2e_n)$ in Hausdorff distance as $\eta \rightarrow 0$.
\end{lemma}

\begin{proof}
	For simplicity, we denote $c_j(\eta)$ by $c_j$.
	Set 
	\begin{align*}
		c_j^0 &=(\eta\omega_j,\sqrt{4-\eta^2}), \\
		x_j^0&=c_j^0-(\omega_j,0)=((\eta-1)\omega_j,\sqrt{4-\eta^2}),
	\end{align*}
	for $j=0,1,\cdots,n$. Note that $c_j^0=c_j$ for $j=1,\cdots,n$.	
	
	Then  $|x_j^0-c_j^0|=1$, i.e., $x_j^0 \in \partial B_1(c_j^0)$, for $0\leq j \leq n$. By (\ref{eqn:Flat-1-1}) and (\ref{eqn:Flat-1-2}), we have that for $0\leq j\ne k\leq n $ and $\eta \in (0,1/2)$, 
	\begin{align*}
		|x_j^0-c_k^0|^2-1
		=-2\eta(1-\eta)(1-\omega_j\cdot\omega_k)<-\eta\,\frac{n-2}{n-1}<0, 
	\end{align*}
	which says $x_j^0\in B_1(c_k^0)$. Since 
	\begin{align*}
		|x_j^0-c_0|\leq |x_j^0-c_0^0|+|c_0^0 -c_0| \leq \sqrt{1-C_n\eta}+ \widetilde{C}_n \eta^2\leq 1-\frac{1}{2}C_n \eta +\widetilde{C}_n \eta^2, \quad j=1,\cdots,n,
	\end{align*}
	there exists $\eta_0(n)\in (0,\frac{1}{2})$ such that $|x_j^0-c_0|<1$ for all $\eta \in (0,\eta_0(n))$. 
	Hence, 
	\begin{align}\label{eqn:Flat-1-3}
		x_j^0\in \partial B_1(c_j), \quad x_j^0 \in \bigcap_{\substack{0\leq k\leq n \\ k\neq j}} B_1(c_k),\quad j=1,\cdots,n.
	\end{align}
	Taking $x_0:=c_0-(\omega_0,0)=x_0^0+c_0-c_0^0$, and 
	\begin{align*}
		|x_0 -c_j| \leq |x_0^0-c_j|+|c_0 - c_0^0 |\leq \sqrt{1-C_n\eta}+ C_n \eta^2<1
	\end{align*}
	then
	\begin{align}\label{eqn:Flat-1-4}
		x_0 \in \partial B_1(c_0), \quad x_0 \in \bigcap_{1\leq k\leq n}B_1(c_k).
	\end{align}
	Combining (\ref{eqn:Flat-1-3}) and (\ref{eqn:Flat-1-4}) gives (ii).

	Let $x_\eta:=(0,\sqrt{4-\eta^2})$, then, $|x_\eta-c_j^0|=\eta$ and $|c_0^0-c_0|
	\leq \widetilde{C}_n \eta^2$, which implies that for all $\eta \in (0,\eta_0(n))$, 
	\begin{align*}
		B_{1-2\eta}(x_\eta)\subset K_\eta.
	\end{align*}	
	Due to $|c_j|=2$ for $0\leq j\leq n$, hence (i) holds.
	
	The centers $c_1,\cdots,c_n$ affinely span one horizontal hyperplane, while $c_0$ has a different last coordinate. Thus, $n+1$ centers are affinely independent and (iii) follows.
	
	Finally, set $d_\eta=\max_{0\leq j\leq n} |c_j-2e_n|$. Triangle inequality gives 
	\begin{align*}
		B_{1-d_\eta}(2e_n) \subset K_\eta \subset B_{1+d_\eta}(2e_n).
	\end{align*}
	Since $d_\eta \rightarrow 0$ as $\eta \rightarrow 0$, this proves Hausdorff convergence.
\end{proof}

From now on, we choose $\eta_*\in (0,\eta_0(n))$ and denote $K_{\eta_*}$ by $\Omega$, i.e., 
\begin{align*}
	\Omega:=K_{\eta_*}=\bigcap_{j=0}^{n}B_1(c_j(\eta_*)).
\end{align*}
In the remainder of paper, we simply denote $c_j(\eta_*)$ by $c_j$, $j=0,\cdots,n$.
Now let us write the solution explicitly.
\begin{proposition}
	Let $a=\frac{n-2}{2}$ and $q_{*}=\frac{n}{n-1}$.  The function
	\begin{equation}\label{eq:explicit-corner-solution}
		u_*(x)=(3a)^a|x|^{2-n}
	\end{equation}
	is a positive, nonconstant, and smooth solution to
	\begin{equation}\label{equ:Robin-K-eta}
		\begin{cases}
			\Delta u=0, & \text{in } \Omega,\\
			\partial_\nu u +a u=u^{q_*}, &\text{a.e. } \partial \Omega,
		\end{cases}
	\end{equation}
	It is therefore also a weak solution of \eqref{equ:Robin-K-eta}.
\end{proposition}

\begin{proof}
	By Lemma~\ref{lem:clustered-faces}, $0\notin \overline{\Omega}$ and $u_* \in C^\infty(\overline{\Omega})$. Direct calculation shows that $u_*$ is harmonic. On the open face contained in
	$\partial B_1(c_j)$, there holds
	\begin{equation}\label{eq:face-geometric-identity}
		\nu=x-c_j,\qquad |x|^2-2x\cdot\nu=|c_j|^2-1=3,
	\end{equation}
	while
	\[
	\partial_\nu u_*
	=-(n-2)\frac{x\cdot\nu}{|x|^2} u_*,
	\qquad
	u_*^{q_*-1}=3a|x|^{-2}.
	\]
	Since $n-2=2a$, $\partial_\nu u +a u=u^{q_*}$ follows immediately from
	\eqref{eq:face-geometric-identity}.  $\Omega$ is Lipschitz and $u_*$
	is smooth in a neighborhood of its closure.  Green's formula on Lipschitz
	domains and the almost-everywhere boundary identity give, for every
	$\varphi\in H^1(\Omega)$,
	\[
	\int_{\Omega}\nabla u_*\cdot\nabla\varphi
	+a\int_{\partial \Omega}u_*\varphi
	=\int_{\partial \Omega}u_*^{q_*}\varphi.
	\]
	This is the required weak formulation.
\end{proof}

The decisive property of $\Omega$ is not merely the explicit
solution, but its nondegeneracy.

Let $\rho(x):=\mathrm{dist}(x,\partial \Omega)$. Then
\[
\rho(x)=1-\max_j|x-c_j|.
\]
Define
\begin{equation}\notag
	V=\rho-\frac{\rho^2}{2}
	=\frac12\left(1-\max_j|x-c_j|^2\right).
\end{equation}
Since $|c_j|=2$, this may also be written as
\begin{equation}\notag
	V(x)=-\frac{|x|^2}{2}
	+\min_j\left(c_j\cdot x-\frac32\right).
\end{equation}
Thus $D^2V\le-I$ in the sense of matrix-valued Radon measures, $V>0$ in
$\Omega$, $V=0$ on $\partial \Omega$, and $\nabla V=-\nu$ on each open spherical
face. In fact, set
\[
G_j=\{x\in \Omega:c_j\cdot x\le c_k\cdot x\ \text{for all }k\},
\qquad
N_{ij}=\frac{c_i-c_j}{|c_i-c_j|},
\]
and let $S_{ij}$ be the relative interior of the $(n-1)$-dimensional part of
$G_i\cap G_j\cap \Omega$; empty interfaces are omitted.  Affine independence
of the centers implies that the locus where three or more affine functions
attain the minimum has codimension at least two.  Hence, the displayed expression
for $V$ gives
\begin{equation}\notag
	D^2V=-I\,\mathcal L^n
	-\sum_{i<j}|c_i-c_j|N_{ij}\otimes N_{ij}\,
	\mathcal H^{n-1}\!\llcorner S_{ij}.
\end{equation}

\begin{lemma}\label{lem:weighted-identity}
	If $v\in H^2(\Omega)$ is harmonic and
	$v_\nu\in L^2(\partial \Omega)$, then
	\begin{equation}\label{eq:weighted-identity}
		\begin{split}
			\int_{\partial \Omega} (\partial_\nu v)^2\dd S
			=&\int_{\Omega}\bigl(|\nabla v|^2+V|D^2v|^2\bigr)\dd x\\
			&+\sum_{i<j}|c_i-c_j|
			\int_{S_{ij}}(\partial_{N_{ij}}v)^2\dd\mathcal H^{n-1}.
		\end{split}
	\end{equation}
\end{lemma}

\begin{proof}
	Since $\Delta (\partial_l v)=0$ distributionally, 
	$V \partial_l v\in H^1_0(\Omega)$ and $V=0$ on $\partial \Omega$, there holds
	\begin{equation}\label{eqn:Flat-2-1}
		\int_{\Omega}V|D^2v|^2=-\int_{\Omega}D^2v(\nabla V,\nabla v). 
	\end{equation}
	On $G_j$, $\nabla V=c_j-x$ and $D^2V=-I$.  The vector field
	$F_j=(\nabla V\cdot\nabla v)\nabla v$ belongs to $W^{1,1}$, and
	\begin{align}\label{eqn:Flat-2-2}
		\diver F_j=-|\nabla v|^2+D^2v(\nabla V,\nabla v).	
	\end{align}
	Since
	\begin{align*}
		\int_{G_j} \mathrm{div} F_j 
		&=\int_{\partial G_j} (\nabla V\cdot \nabla v ) \partial_\nu v \\
		&= \int_{\partial G_j \cap \partial \Omega} (\nabla V\cdot \nabla v ) \partial_\nu v
		+ \sum_{k\neq j} \int_{S_{jk}} (\nabla V\cdot \nabla v ) \partial_\nu v \\
		&= -\int_{\partial G_j \cap \partial \Omega} (\partial_\nu v)^2
		+ \sum_{k\neq j} \int_{S_{jk}} ((c_j-x)\cdot \nabla v ) \partial_{N_{jk}} v,	
	\end{align*}
	we obtain
	\begin{align*}
		\int_{\Omega} \mathrm{div} F_j = \sum_j \int_{G_j} \mathrm{div} F_j
		&=-\int_{\partial \Omega} (\partial_\nu v)^2 + \sum_{j<k} \int_{S_{jk}} ((c_j-c_k)\cdot \nabla v ) \partial_{N_{jk}} v \\
		&=-\int_{\partial \Omega} (\partial_\nu v)^2 + \sum_{j<k} \int_{S_{jk}} |c_j-c_k| (\partial_{N_{jk}} v)^2 
	\end{align*}
	which combining (\ref{eqn:Flat-2-1}) and (\ref{eqn:Flat-2-2}) gives (\ref{eq:weighted-identity}).
\end{proof}

\begin{proposition}\label{prop:corner-steklov-gap}
	The first nonzero Steklov eigenvalue of $\Omega$ satisfies
	\begin{equation}\notag
		\sigma_1(\Omega)>1.
	\end{equation}
\end{proposition}

\begin{proof}
	Let $\sigma>0$ be a Steklov eigenvalue and let $v$ be a nonzero eigenfunction,
	so $\partial_\nu v=\sigma v$ and $v|_{\partial \Omega}\neq 0$.  Weak Steklov eigenfunctions on bounded convex domains
	belong to $H^2$ (see \cite{LambertiProvenzano2026} Theorem~1); hence
	Lemma~\ref{lem:weighted-identity} applies and there holds
	\begin{align}\label{ineq:First-Steklov-1}
		\int_{\partial \Omega} (\partial_\nu v)^2\dd S
		\geq \int_{\Omega}\bigl(|\nabla v|^2+V|D^2v|^2\bigr)\dd x.
	\end{align}
	The identities
	\[
	\int_{\Omega}|\nabla v|^2 = \int_{\partial \Omega} v \partial_\nu v
	=\sigma\int_{\partial \Omega}v^2,
	\qquad
	\int_{\partial \Omega} (\partial_\nu v)^2
	=\sigma^2\int_{\partial \Omega}v^2
	\]
	imply $\sigma\ge1$.
	
	If $\sigma=1$, \eqref{ineq:First-Steklov-1} forces $D^2v=0$, so $v(x)=p\cdot x+b$.  On the
	$j$th open face, $\partial_\nu v=v$ becomes $p\cdot(x-c_j)=p\cdot x+b$ which leads to
	\[
	\qquad p\cdot c_j=-b.
	\]
	Every face is nonempty and the centers are affinely independent; hence
	$p=0$ and $b=0$, a contradiction occurs and all the Steklov eigenvalues $\sigma>1$. In particular, $\sigma_1(\Omega)>1$.
\end{proof}

We will show that the gap $\sigma_1(\Omega)>1$ can be transferred to the linearization at $u_*$ by a Kelvin involution
that preserves all the unit spheres.  Set
\begin{equation}\notag
	T(x)=\frac{3x}{|x|^2},\qquad
	(\mathcal Kv)(x)=\left(\frac{\sqrt3}{|x|}\right)^{n-2}v(Tx).
\end{equation}
A direct computation gives
\begin{equation}\notag
	|T(x)-c_j|^2-1
	=\frac3{|x|^2}\bigl(|x-c_j|^2-1\bigr),
\end{equation}
so $T(\Omega)=\Omega$.  Since $0\notin \overline{\Omega}$ and $\mathcal{K}^2=\mathrm{Id}$, $\mathcal K$ is a bounded involution on $H^1(\Omega)$.  The boundary operator $B:=\partial_\nu+a$ obeys
\begin{equation*}
	B(\mathcal Kv)(x)
	=\left(\frac{\sqrt3}{|x|}\right)^n(Bv)(Tx),
\end{equation*}
which implies that if $B(v)=v^{q_*}$, then $B(\mathcal{K}v)=(\mathcal{K}v)^{q_*}$. And the Laplace operator satisfies
\begin{align*}
	\Delta (\mathcal{K}v)=\left(\frac{\sqrt{3}}{|x|}\right)^{n+2} (\Delta v)(Tx).
\end{align*}
Since
\begin{align*}
	&\int_{\Omega} \nabla (\mathcal{K}v)\nabla (\mathcal{K}\phi) 
	+a \int_{\partial \Omega} (\mathcal{K}v)(\mathcal{K}\phi) -q_* \int_{\partial \Omega} (\mathcal{K}u_*)^{q_*-1} (\mathcal{K}v)(\mathcal{K}\phi)\\
	=& \int_{\Omega} \nabla v \nabla \phi +a \int_{\partial \Omega} v \phi -q_* \int_{\partial \Omega} u_*^{q_*-1} v\phi,
\end{align*}
for all $v,\phi \in H^1(\Omega)$, and 
\begin{equation*}
	\mathcal Ku_*=a^a,
\end{equation*}
$\mathcal{K}$ transforms the linearized problem at $u_*$ into the Steklov eigenvalue problem on $\Omega$
\begin{align*}
	\begin{cases}
		\Delta v=0 \quad \text{in } \Omega\\
		\partial_\nu v=v \quad \text{on } \partial \Omega.
	\end{cases}
\end{align*}
Proposition~\ref{prop:corner-steklov-gap}
implies the following conclusion.
\begin{corollary}\label{cor:corner-nondegenerate}
	The linearization of \eqref{equ:Robin-K-eta} at $u_*$ has trivial kernel in $H^1(\Omega)$.
\end{corollary}

\section{Positive smooth solutions on approximating smooth convex domains }
\label{sec:solution-approximating-domain}

The domain $\Omega$ supplies both an explicit solution and strict
nondegeneracy, but its boundary is not smooth.  We approximate it by smooth
domains while preserving the principal-curvature lower bound bigger than $1$ and prove that
the nondegenerate solution persists.

Let $f_j(x)=|x-c_j|^2-1$ and define
\begin{equation}\notag
	G_\tau(x)=
	\tau\log\left(2\sum_{j=0}^{n}e^{f_j(x)/\tau}\right),
	\qquad
	\Omega_\tau=\{x\in \mathbb{R}^n: G_\tau(x)<0\}.
\end{equation}
As $\tau\downarrow0$, $G_\tau$ converges to $\max_jf_j$, so
$\Omega_\tau$ is the standard inner smoothing of the ball intersection $\Omega$.

Set
\begin{equation}\notag
	\pi_j=\frac{e^{f_j/\tau}}{\sum_ke^{f_k/\tau}},
	\qquad
	\bar c=\sum_{j=0}^n \pi_jc_j.
\end{equation}
Direct differentiation gives
\begin{equation}\label{eq:softmax-hessian}
	\nabla G_\tau=2(x-\bar c),
	\qquad
	D^2G_\tau
	=2I+\frac4\tau
	\sum_{j=0}^n \pi_j(c_j-\bar c)\otimes(c_j-\bar c).
\end{equation}

\begin{proposition}\label{prop:softmax-geometry}
	Let
	\begin{equation}\notag
		\rho_\tau=\sqrt{1-\tau\log2}.
	\end{equation}
	There exists $\mu_0>0$ only dependent of $n$ and $\Omega$ such that for every $\tau\in (0,\mu_0]$, the domain $\Omega_\tau$ is
	nonempty, bounded, real analytic, and strictly convex, and
	\begin{equation}\label{ineq:principle-curvature-1}
		\min_i\kappa_i(\partial \Omega_\tau)>\rho_\tau^{-1}>1.
	\end{equation}
\end{proposition}

\begin{proof}
	Equation \eqref{eq:softmax-hessian} gives $D^2G_\tau\ge2I$.  The common
	interior point $x_{\eta_*}$ in Lemma~\ref{lem:clustered-faces} satisfies
	$\max_jf_j(x_{\eta_*})<0$, hence $G_\tau(x_{\eta_*})<0$ for small $\tau$; meanwhile,
	$G_\tau(x)\to+\infty$ as $|x|\to\infty$.  Strict convexity makes the zero
	level set regular, compact, and strictly convex.  Since $G_\tau$ is real
	analytic, so is its zero level.
	
	On $\partial \Omega_\tau$,
	$2\sum_je^{f_j/\tau}=1$ which implies
	\begin{equation*}
		|x-c_j|<\rho_\tau\qquad j=0,1,\cdots,n.
	\end{equation*}
	Therefore, on $\partial \Omega_\tau$, 
	\[
	0<|\nabla G_\tau|
	=2\left|\sum_{j=0}^n \pi_j(x-c_j)\right|<2\rho_\tau.
	\]
	For every unit tangent vector $X$ on $\partial\Omega_\tau$, there holds
	then
	\[
	\II(X,X)=\frac{D^2G_\tau(X,X)}{|\nabla G_\tau|}
	>\rho_\tau^{-1}
	\]
	and (\ref{ineq:principle-curvature-1}) follows.
\end{proof}

To compare equations on $\Omega$ and on the smooth domains $\Omega_\tau$ in one function
space, we construct an explicit radial pullback.  The log-sum-exp bounds give
\begin{equation}\notag
	\max_jf_j+\tau\log2
	\le G_\tau
	\le\max_jf_j+\tau\log(2(n+1)).
\end{equation}
After decreasing $\tau$ so that $1-\tau\log(2(n+1))>0$, set
\[
\underline\rho_\tau=\sqrt{1-\tau\log(2(n+1))}.
\]
Then
\begin{equation}\notag
	\bigcap_j B_{\underline\rho_\tau}(c_j)
	\subset \Omega_\tau
	\subset\bigcap_j B_{\rho_\tau}(c_j)
	\subset \Omega.
\end{equation}
Set $d_\eta:=\max_j|x_\eta-c_j|<1$, define
\[
\vartheta_\tau=\frac{1-\underline\rho_\tau}{1-d_\eta},
\qquad
T_\tau x=(1-\vartheta_\tau)x+\vartheta_\tau x_\eta,
\]
it follows that
\begin{equation}\label{eq:homothetic-sandwich}
	T_\tau(\overline \Omega)
	\subset \bigcap_j \overline{B_{\underline\rho_\tau}(c_j)}
	\subset\overline \Omega_\tau
	\subset\overline \Omega.
\end{equation}

We retain the original coordinate origin, which is the pole of
$u_*(x)=(3a)^a|x|^{2-n}$.  With the common interior point $x_{\eta_*}$ as radial
center, let $\gamma_\tau$ and $\gamma_0$ be the Minkowski functionals of
$\Omega_\tau-x_{\eta_*}$ and $\Omega-x_{\eta_*}$.

By \eqref{eq:homothetic-sandwich}, there exists $0<r_*<R_*$ such that for sufficient small $\tau>0$ such that $$B_{r_*}(x_{\eta_*})\subset \Omega_\tau\subset \Omega \subset B_{R_*}(x_{\eta_*})$$
which give uniform positive bounds and Lipschitz bounds for these homogeneous convex functions on every fixed annulus.  Hausdorff convergence implies
\[
\gamma_\tau\longrightarrow\gamma_0
\quad\text{locally uniformly on }\R^n\setminus\{0\}.
\]

At a differentiability point $x$ of $\gamma_0$, take any convergent
subsequence of $\nabla\gamma_\tau(x)$.  Passing to the limit in
\[
\gamma_\tau(y)\ge
\gamma_\tau(x)+\nabla\gamma_\tau(x)\cdot(y-x)
\]
shows that the limit belongs to
$\partial\gamma_0(x)=\{\nabla\gamma_0(x)\}$.  Hence the gradients converge
almost everywhere.  The common Lipschitz bound and dominated convergence
then give
\[
\gamma_\tau\to\gamma_0
\quad\text{in }W^{1,r}_{\mathrm{loc}}(\R^n\setminus\{0\})
\qquad(1\le r<\infty).
\]

Write the radial graphs as
\[
\partial \Omega=\{x_{\eta_*}+R_0(\theta)\theta:\theta\in S^{n-1}\},\qquad
\partial \Omega_\tau
=\{x_{\eta_*}+R_\tau(\theta)\theta:\theta\in S^{n-1}\}.
\]
Since $R_\tau=1/\gamma_\tau$ on the sphere,
\begin{equation}\label{eq:radial-convergence}
	R_\tau\longrightarrow R_0
	\quad\text{in }W^{1,r}(S^{n-1})
	\qquad(1\le r<\infty).
\end{equation}
Set $\beta_\tau=R_\tau/R_0$ and define $\Phi_\tau :\Omega \rightarrow \Omega_\tau$, along each ray,
\begin{equation}\notag
	\Phi_\tau(x_{\eta_*}+t\theta)
	=x_{\eta_*}+t\beta_\tau(\theta)\theta, \quad t\in [0,R_0(\theta)).
\end{equation}
The functions $\beta_\tau$ and $\beta_\tau^{-1}$ have common positive
bounds and Lipschitz bounds.  More explicitly, at almost every point, after
identifying a tangent vector with its radial and spherical components,
\begin{equation}\label{eq:radial-map-derivative}
	D\Phi_\tau
	=\beta_\tau I+\theta\otimes\nabla_S\beta_\tau,
	\qquad
	\det D\Phi_\tau=\beta_\tau^n,
\end{equation}
where $\nabla_S$ is the Levi-Civita connection on $S^{n-1}$.
The inverse has the same form with $\beta_\tau$ replaced by
$\beta_\tau^{-1}$ on the target ray.  The common bounds just noted therefore
make $\Phi_\tau:\Omega\to \Omega_\tau$ uniformly bi-Lipschitz.
Moreover, \eqref{eq:radial-convergence} gives
$\beta_\tau\to1$ and $\nabla_S\beta_\tau\to0$ in $L^r(S^{n-1})$ ($1\leq r<\infty$);
hence $D\Phi_\tau\to I$ in $L^r(\Omega)$ ($1\leq r<\infty$).

Pulling the Dirichlet energy back to $\Omega$, set
\begin{equation}\notag
	A_\tau=|\det D\Phi_\tau|\,
	D\Phi_\tau^{-1} (D\Phi_\tau^{-1})^T,
	\qquad J_\tau=J_{\partial\Phi_\tau}.
\end{equation}
These matrices are uniformly bounded and elliptic, while $J_\tau$ has common
positive upper and lower bounds.  The radial-graph area formula is
\[
J_\tau(x_{\eta_*}+R_0(\theta)\theta)=
\frac{R_\tau^{\,n-2}
	\sqrt{R_\tau^2+|\nabla_SR_\tau|^2}}
{R_0^{\,n-2}
	\sqrt{R_0^2+|\nabla_SR_0|^2}},
\]
and hence, for all $1\leq r<\infty$,
\begin{equation}\label{eq:coefficient-convergence}
	A_\tau\to I\quad\text{in }L^r(\Omega),
	\qquad
	J_\tau\to1\quad\text{in }L^r(\partial \Omega).
\end{equation}
Indeed, the assertions for $A_\tau$ follow directly from
\eqref{eq:radial-map-derivative}: its entries are rational expressions in
$\beta_\tau$ and $\nabla_S\beta_\tau$, with denominators bounded away from
zero.  The displayed area formula gives the corresponding boundary
statement. 

In order to obtain the compactness of resolvents of problem \eqref{eq:wang-problem-Omega}, let us recall some facts about Robin boundary problem related to our problems.
\begin{proposition}\label{prop:Robin-bdy-g-Omega}
	Suppose that $U \subset \mathbb{R}^n$ is a bounded domain with Lipschitz boundary. Then for any $g\in C(\partial U)$, the Robin boundary problem
	\begin{equation}\label{prob:Robin-Omega}
		\begin{cases}
			\Delta u=0, & \text{in } U,\\
			\partial_\nu u +a u=g &\text{on } \partial U,
		\end{cases}
	\end{equation}
	admits a unique weak solution $u\in H^1(U)$. Moreover, there exists $\alpha \in (0,1)$ and $C>0$, both only dependent of $n$ and $U$ such that
	\begin{align}
		\|u\|_{H^1(U)} &\leq C \|g\|_{L^\infty(\partial U)} \label{eqn:Flat-5-1} \\
		\|u\|_{L^\infty(U)} &\leq \frac{1}{a} \|g\|_{L^\infty(\partial U)} \label{eqn:Flat-5-2}\\
		\|u\|_{C^{0,\alpha}(\overline{U})} &\leq C \|g\|_{L^\infty(\partial U)} \label{eqn:Flat-5-3}
	\end{align}
\end{proposition}
The existance and uniqueness of weak solutions to \eqref{prob:Robin-Omega} follows from classic elliptic theory. Recall the weak equation of Robin boundary problem 
\begin{align*}
	\int_{U} \nabla u \nabla v+ a\int_{\partial U} u v 
	-\int_{\partial U} g v=0, \quad \forall v\in H^1(U).
\end{align*}
Testing with $v=u$ and using the Young's inequality and trace-Poincare inequality, we obtain \eqref{eqn:Flat-5-1}. Testing with $v=(u-a^{-1}\|g\|_{L^\infty(\partial U)})_+$ and $v=(-u-a^{-1}\|g\|_{L^\infty(\partial U)})_+$ also gives \eqref{eqn:Flat-5-2}. Finally, \eqref{eqn:Flat-5-3} follows from \cite[Theorem~3.14]{Nittka2011}.

Let $X:=C(\partial \Omega)$. Define the operator $\mathcal{R}_\tau :X \rightarrow X$ as follows: Given $g\in X$, solve
\begin{equation}\notag
	\begin{cases}
		\Delta \omega=0, & \text{in } \Omega_\tau,\\
		\partial_\nu w+a w=g\circ\Phi_\tau^{-1}, &\text{on } \partial \Omega_\tau,
	\end{cases}
\end{equation}
and obtain the solution $\omega_\tau$; write $v_\tau=w_\tau\circ\Phi_\tau$ and let $\mathcal{R}_\tau (g):= v_\tau |_{\partial \Omega}$. For $\tau=0$, define it directly on
$\Omega$. By Proposition \ref{prop:Robin-bdy-g-Omega}, $\mathcal{R}_\tau$ is well-defined. We will show that $\mathcal{R}_\tau$ are compact linear operators for sufficiently small $\tau$.

\begin{lemma}\label{lem:robin-continuity}
	There exist $\alpha\in (0,1)$ and $C>0$ such that for all $\tau \in [0,\mu_0]$ ($\mu_0$ comes from Proposition \ref{prop:softmax-geometry})
	\begin{align}\label{eq:uniform-holder}
		\|v_\tau\|_{H^1(\Omega)} + \|v_\tau\|_{L^\infty(\Omega)}+ \|v_\tau\|_{C^{0,\alpha}(\overline{\Omega})} \leq C \|g\|_{L^\infty(\partial \Omega)}
	\end{align}
	Moreover, there holds
	\begin{enumerate}
		\item[(i)] $\mathcal{R}_\tau$ are compact for all $\tau\in [0,\mu_0]$;
		\item[(ii)] If $\tau_k \rightarrow 0$ and $g_k \rightarrow g$ in $X$, then 
		\begin{align}\label{eq:robin-operator-convergence}
			\mathcal{R}_{\tau_k} g_k\rightarrow \mathcal{R}_0 g \quad \text{in }X;
		\end{align}
		\item[(iii)] If $\tau_k \rightarrow 0$ and $g_k$ is bounded in $X$, then $\{ \mathcal{R}_{\tau_k}g_k\}$ has a convergent subsequence in $X$.
	\end{enumerate}
\end{lemma}

\begin{proof}
	The pulled-back weak equation is
	\begin{equation}\label{eq:pulled-robin-weak-form}
		\int_{\Omega}A_\tau\nabla v_\tau\cdot\nabla\phi
		+a\int_{\partial \Omega}J_\tau v_\tau\phi
		=\int_{\partial \Omega}J_\tau g\phi, \quad \forall \phi \in H^1(\Omega).
	\end{equation}
	Recall that $\Phi_\tau:\Omega \rightarrow \Omega_\tau$ is uniform bi-Lipschitz, $A_\tau$ are uniformly bounded and elliptic, and $J_\tau$ has common positive upper and lower bounds. By Proposition \ref{prop:Robin-bdy-g-Omega}, we can obtain \eqref{eq:uniform-holder}.
	
	Let $v_k$ correspond to $(\tau_k,g_k)$. If $g_k$ is bounded in $X$, it follows from \eqref{eq:uniform-holder} that $\{v_k\}$ is bounded in $H^1(\Omega)$ and $C^{0,\alpha}(\overline{\Omega})$. Hence there exist subsequence, for simplicity, also denoted by $v_k$ such that 
	\begin{align*}
		v_k \rightharpoonup v \quad \text{in } H^1(\Omega), \quad
		v_k \rightarrow v \quad \text{in } C(\overline{\Omega}),
		\quad \mathcal{R}_{\tau_k}g_k=v_k|_{\partial \Omega} \rightarrow v|_{\partial \Omega} \quad \text{in } C(\partial \Omega)
	\end{align*}
	Hence, (iii) follows. If $\tau$ is fixed, we can get (i).  Moreover, if $g_k \rightarrow g$ in $X$, by \eqref{eq:coefficient-convergence}, there holds, for all $\phi \in H^1(\Omega)$,
	\begin{align*}
		\int_{\Omega}\nabla v\cdot\nabla\phi
		+a\int_{\partial \Omega} v \phi
		&=\lim_{k\rightarrow \infty}\int_{\Omega}A_{\tau_k}\nabla v_k\cdot\nabla\phi
		+a\int_{\partial \Omega}J_{\tau_k} v_k\phi \\
		&=\lim_{k\rightarrow \infty} \int_{\partial \Omega} J_{\tau_k} g_k\phi
		=\int_{\partial \Omega} g \phi.
	\end{align*}
	Thus, $v|_{\partial \Omega}=\mathcal{R}_0 g$. Due to the uniquness of $v$, every subsequence has the same limit, (ii) is proved.
	
\end{proof}

\begin{proposition}\label{prop:softmax-persistence}
	Let $\mathcal{B} \subset X$ be any sufficiently small open ball centered at
	$z_*=u_*|_{\partial \Omega}$ whose closure contains neither another solution
	of \eqref{equ:Robin-K-eta} nor a constant function.  There are $\mu_1\in (0,\mu_0)$
	and $0<\varepsilon<q_*-1$ such that, for any
	\[
	0<\tau<\mu_1,\qquad  q_*-\varepsilon <q\leq q_*,
	\]
	the problem
	\begin{equation}\label{eq:unscaled-smooth-problem}
		\Delta u=0\quad\text{in }\Omega_\tau,\qquad
		\partial_\nu u+\frac1{q-1}u=u^q
		\quad\text{on }\partial \Omega_\tau
	\end{equation}
	has a positive nonconstant smooth solution $u_{\tau,q}$ such that $(u_{\tau,q}\circ \Phi_\tau) |_{\partial \Omega} \in \mathcal{B}$. The solutions may be selected so that 
	$(u_{\tau_k,q_k}\circ \Phi_\tau) |_{\partial \Omega} \rightarrow u_*|_{\partial \Omega}$ in $X$ as $(\tau_k,q_k)\to(0,q_*)$.
\end{proposition}

\begin{proof}
	Corollary~\ref{cor:corner-nondegenerate} says that $z_*$ is an isolated
	nondegenerate fixed point.  Shrink $\mathcal{B}\subset X$, as allowed in the
	statement, so that every function in its closure takes values in some
	$[m_0,M_0]\subset(0,\infty)$.  Choose a cut-off function 
	$\psi\in C_c^\infty((0,\infty))$ equal to $1$ on $[m_0,M_0]$, and let
	\[
	\chi_q(t)=
	\begin{cases}
		\psi(t)t^q,&t>0,\\
		0,&t\le0.
	\end{cases}
	\]
	Define
	the compact map on $X$
	\begin{equation}\notag
		\mathcal F_{\tau,q}(z)=
		\mathcal R_\tau\left(
		\chi_q(z)+\left(a-\frac1{q-1}\right)z
		\right).
	\end{equation}
	If there exist $\hat{z}\in \mathcal{B}$ such that $\mathcal F_{\tau,q}(\hat{z})=\hat{z}$, there exists a solution $\hat{u}$ of \eqref{eq:unscaled-smooth-problem} such that $\hat{z}=(\hat{u} \circ \Phi_\tau)|_{\partial \Omega}$.
	Moreover, by Proposition \ref{prop:Robin-bdy-g-Omega} and $\hat{z}\in \mathcal{B}$, there holds 
	\begin{align*}
		\|\hat{u}\|_{H^1(\Omega_\tau)} + \|\hat{u}\|_{L^\infty(\Omega_\tau)}+ \|\hat{u}\|_{C^{0,\alpha}(\overline{\Omega_\tau})} \leq C.
	\end{align*}
	Then classic elliptic Neumann boundary regularity on the smooth domain yields $\hat{u}\in C^\infty(\overline \Omega_\tau)$.
	Hence, the existence problem of \eqref{eq:unscaled-smooth-problem} is reduced to existence of fixed-points of $\mathcal F_{\tau,q}$.

	At $(\tau,q)=(0,q_*)$, $z_*$ is a fixed point.
	Corollary~\ref{cor:corner-nondegenerate} shows that
	$I-D_z\mathcal F_{0,q_*}(z_*)$ has trivial kernel. Thus, after shrinking $\mathcal{B}$ if needed, there holds
	\begin{equation}\label{eq:local-degree-nonzero}
		\deg(I-\mathcal F_{0,q_*},\mathcal{B},0)=\pm1
	\end{equation}
	by \cite[\S\S8.3--8.4]{Deimling1985}.
	
	We claim that there exist sufficiently small $\mu_1 \in (0,\mu_0)$ and $\varepsilon\in (0,q_*-1)$ such that for all $\tau\in (0,\mu_1)$ and $q\in (q_*-\varepsilon,q_*]$, the straight-line homotopy
	\begin{align}
		H(s,(\tau,q),z):=&(1-s)(I- \mathcal{F}_{0,q_*})(z)+s (I-\mathcal F_{\tau,q})(z) \notag \\
		=&z-\bigg((1-s) \mathcal{F}_{0,q_*}(z)+s \mathcal F_{\tau,q}(z) \bigg), \quad s\in [0,1]
		\label{eq:straight-line-homotopy}
	\end{align}
	satisfies $0\notin H(s,(\tau,q),\partial \mathcal{B})$ for $s\in [0,1]$. Otherwise
	there would be $\tau_k\to0$, $q_k\to q_*$, $s_k\in[0,1]$, and
	$z_k\in\partial \mathcal{B}$ satisfying
	\begin{equation}\label{eq:degree-homotopy}
		z_k=(1-s_k)\mathcal F_{0,q_*}(z_k)
		+s_k\mathcal F_{\tau_k,q_k}(z_k).
	\end{equation}
	Since $\{ z_k\}$ is bounded in $X$, it follows from Lemma~\ref{lem:robin-continuity} that $\{F_{0,q_*}(z_k)\}$ and $\{\mathcal F_{\tau_k,q_k}(z_k)\}$ are relatively compact in $X$ which implies $\{ z_k\}$ is also relatively compact in $X$ by \eqref{eq:degree-homotopy} . Hence, after passing to a subsequence, $z_k\to z$ and $s_k\to s$. By \eqref{eq:robin-operator-convergence}, taking limit in \eqref{eq:degree-homotopy}, we obtain $z=\mathcal F_{0,q_*}(z)$ with $z\in \partial \mathcal{B}$, which contradicts that $z_*$ is an isolated fixed point in $\mathcal{B}$.
	
	Hence, It follows from above claim that the degree for $I-\mathcal{F}_{\tau,q}$ remains \eqref{eq:local-degree-nonzero} and $\mathcal{F}_{\tau,q}$ has a fixed point for every $(\tau,q)\in (0,\mu_1)\times (q_*-\varepsilon, q_*]$. The choice of neighborhood $\mathcal{B}$ makes them positive and nonconstant.  Applying the similar process in above contradiction argument shows that solutions in $\mathcal{B}$ may be selected with traces converging to
	$z_*$ as $(\tau,q)\to(0,q_*)$.
\end{proof}

Suppose $u$ is a solution of \eqref{eq:unscaled-smooth-problem} in Proposition \ref{prop:softmax-persistence}. Then for
$1\leq r\le\rho_\tau^{-1}$ define
\begin{equation}\notag 
	\Omega_{\tau,r}=r\Omega_\tau,\qquad
	u_r(y)=r^{-\frac{1}{q-1}} \, u \left(\frac{y}{r}\right).
\end{equation}
This interval is nonempty because $0<\rho_\tau<1$.
The function $u_r$ solves
\begin{equation*}
	\Delta u=0\quad\text{in }\Omega_{\tau,r},\qquad
	\partial_\nu u+ \lambda u=u^q
	\quad\text{on }\partial \Omega_{\tau,r},
\end{equation*}
where 
\begin{align*}
	\lambda=\frac{1}{r}\cdot \frac{1}{q-1}.
\end{align*}
Note that $\frac{1}{r}\in [\rho_\tau,1]$.
Proposition~\ref{prop:softmax-geometry}
therefore gives
\[
\min_i \kappa_i(\partial\Omega_{\tau,r})
=\frac{1}{r}\min_i \kappa_{i}(\partial\Omega_{\tau})
>\frac1{r\rho_\tau}\ge1.
\]
The sharp Steklov estimate under nonnegative sectional curvature
\cite[Theorem~2]{XiaXiong2024} gives in addition
\[
\sigma_1(\Omega_{ \tau,r})
>\frac1{r\rho_\tau}\geq 1.
\]
Hence, given $\bar{\tau}\in (0,\mu_1)$, taking $\varepsilon$ in Proposition \ref{prop:softmax-persistence} and $\delta:=\frac{1}{2}(\rho_{\bar{\tau}}+1)$, for every
$q_*-\varepsilon<q\leq q_*$ and $\frac{1}{q-1}\cdot\delta <\lambda \leq \frac{1}{q-1}$, we choose $\Omega_*=\Omega_{\bar{\tau},1/(\lambda(q-1))}$ and Theorem~\ref{thm:flat-main} is proved.

\vspace{1cm}

\noindent \textbf{Disclosure on AI assistance.} The authors used AI-assisted tools, principally ChatGPT 5.6.
The authors verified and completed all mathematical arguments, and take full responsibility for its content.


\end{document}